\documentclass[review]{elsarticle}
\usepackage{amsthm}
\usepackage{amsmath}
\usepackage{graphicx,setspace}
\usepackage{amsfonts,amsmath, amssymb, amsthm}
\usepackage{mathrsfs}
\usepackage{epstopdf}
\usepackage{float}
\usepackage{lineno,hyperref}
\usepackage{color}
\usepackage{subfigure}
\usepackage{bm}
\usepackage{graphicx}
\usepackage{amssymb, amsthm}
\usepackage{mathrsfs}
\usepackage{epstopdf}
\usepackage{float}
\usepackage{lineno,hyperref}
\usepackage{color}
\usepackage{subfigure}
\usepackage{cases}
\usepackage{amsmath,amsthm}
\modulolinenumbers[5]
\numberwithin{equation}{section}

\journal{}

\begin{document}

\newtheorem{definition}{Definition}
\newtheorem{lemma}{Lemma}
\newtheorem{remark}{Remark}
\newtheorem{theorem}{Theorem}

\newtheorem{proposition}{Proposition}
\newtheorem{assumption}{Assumption}
\newtheorem{example}{Example}
\newtheorem{corollary}{Corollary}
\def\e{\varepsilon}
\def\Rn{\mathbb{R}^{n}}
\def\Rm{\mathbb{R}^{m}}
\def\mE{\mathbb{E}}
\def\hte{\bar\theta}
\def\cC{{\mathcal C}}
\numberwithin{equation}{section}

\begin{frontmatter}

\title{{\bf Effective filtering analysis for non-Gaussian dynamic systems}
\footnote{
This work was partly supported by the NSF grant 1620449,  and NSFC grants 11371352, 11531006 and 11771449.}}
\author{\centerline{\bf Yanjie Zhang$^{a,}\footnote{ zhangyanjie2011@163.com .}$,
Huijie Qiao$^{b,*}\footnote{corresponding author: hjqiaogean@seu.edu.cn.}$ and
Jinqiao Duan$^{c,}\footnote{duan@iit.edu}$}
\centerline{${}^a$ Center for Mathematical Sciences \& School of Mathematics and Statistics,}
\centerline{Huazhong University of Sciences and Technology, Wuhan 430074,  China}
\centerline{${}^b$ School of Mathematics,} \centerline{Southeast University, Nanjing, Jiangsu 211189,  China}
\centerline{${}^c$ Department of Applied Mathematics,} \centerline{Illinois Institute of Technology, Chicago, IL 60616, USA}}

\begin{abstract}
This work is about a slow-fast data assimilation system under non-Gaussian noisy fluctuations. Firstly, we show the existence of a random invariant manifold for a  stochastic dynamical system  with  non-Gaussian noise and  two-time scales. Secondly, we obtain a
low dimensional reduction of this system via a random invariant manifold. Thirdly, we prove that the low dimensional filter on the random invariant manifold approximates the original filter, in a probabilistic sense.
\end{abstract}


\end{frontmatter}

\section{Introduction}
Data assimilation is a procedure to extract system state information with help of noisy observations \cite{BC, ge, RBL}. The state evolution is often non-Gaussian (in particular, L\'evy type) in complex  systems.   Data assimilation systems under non-Gaussian noisy fluctuations  have been recently considered by us and other authors \cite{cckc2, qiao, sun, spss}.  It is also desirable to  consider  data assimilation when the system evolution is under  L\'evy motions with two-time scales.

 Stochastic dynamical systems   with non-Gaussian noise with two-time scales arise widely in mathematical modeling \cite{sun, gid, bg, Fu}. Treating stochastic differential equations (SDEs) with two-time scales,  Khasminskii and Yin \cite{KY} developed a stochastic averaging principle that enables   to average out the fast varying variables. Recently, we have further showed  that the averaged, low dimensional formal filter approximates the original filter, by examining the corresponding Zakai stochastic partial differential equations \cite{zhang}.

 The theory of random invariant manifolds provides a geometric approach for eliminating the fast variables by a fixed point technique \cite{Fu, Kai}.
 For slow-fast SDEs  with  Gaussian noise on the long time scales, qualitative analysis for the behavior of random invariant manifold  can be found in Wang and Roberts \cite{WR}. Filtering problems on a random invariant manifold driven by Gaussian noise have been studied \cite{HJ}. In this present paper, our goal is to investigate filtering for stochastic differential equations   with $\alpha$-stable L\'evy noise and two-time scales, when we can only observe the slow component. Firstly, we show the existence of  a random invariant manifold for a stochastic dynamical system  driven by $\alpha$-stable noise with two-time scales. Secondly, we obtain its
low dimensional reduction via the random invariant manifold. Thirdly, we prove that the low dimensional filter on the random invariant manifold approximates the original filter, as the scale parameter tends to zero.


This paper is organized as follows. In Section 2,  we recall basic concepts about symmetric $ \alpha$ -stable L\'evy process and random dynamical systems. Section 3 is devoted to the existence of a random invariant manifold for a slow-fast stochastic system with  $ \alpha$ -stable L\'evy process. In Section 4, we present
a nonlinear filtering problem this slow-fast system, and prove that the low dimensional   filter approximates the original filter on the  random invariant manifold. Finally, we end the paper with a summary on our results and  some   discussions.

The following convention will be used throughout the paper: $C$    denotes different positive constants  whose values may change from one place to another; $\mathcal{C}(\mathbb{R}^n)$   is the set of all real-valued continuous functions on $\mathbb{R}^n$; we denote by $\mathcal{C}_b^1(\mathbb{R}^n)$  the set of all continuous functions on $ \mathbb{R}^{n}$, with first-order derivatives   are   uniformly bounded.  For $\phi \in \mathcal{C}_b^1(\mathbb{R}^n)$, we introduce the following norm
\begin{equation}
||\phi||={\mathbb{\max}}_{{x\in \mathbb{R}^{n}}}|\phi(x)|+{\mathbb{\max}}_{{x\in \mathbb{R}^{n}}}|\nabla\phi(x)|,
\end{equation}
where $\nabla$ represents the gradient operator.

Let $ Lip(\mathbb{R}^{m}, \mathbb{R}^{n})$ denote the set of globally Lipschitz continuous functions $ \gamma $ mapping $ \mathbb{R}^{m} $ into $ \mathbb{R}^{n} $ with the following semi-norm
\begin{equation*}
||\gamma||_{Lip}:=\sup_{y_1\neq y_2 \in \mathbb{R}^{m}}\frac{|\gamma(y_1)-\gamma(y_2)|}{|y_1-y_2|} \textless \infty,
\end{equation*}
and $ \mathcal{L}$ the subset of $ Lip(\mathbb{R}^{m},\mathbb{R}^{n} )$ consisting of bounded functions  with the following norm
\begin{equation*}
||\gamma||_{\infty}:=\sup_{\tilde{y}\in \mathbb{R}^{m}}|\gamma(\tilde{y})|.
\end{equation*}
The set $\mathcal{L}_{\kappa}\subseteq  \mathcal{L}$ whose each element $\gamma$ satisfies $||\gamma||_{Lip}\leq \kappa$.

\section{Preliminaries}
In this section, we recall some basic definitions for L\'evy motions (or L\'evy processes) and random dynamical systems.
\subsection{\textbf{Symmetric $\alpha $ -stable L\'evy processes }(see \cite{da, Duan})}
\begin{definition}
A stochastic process $L_t$ is a  L\'evy  process if
\begin{enumerate}
\item[(1)]$L_0=0$  (a.s.);
\item[(2)]$L_t$ has independent increments and stationary increments; and
\item[(3)] $L_t$ has stochastically continuous sample paths, i.e.,  for  every  $s\geq 0$,
 $L(t)\rightarrow L(s)$ in probability, as $t \rightarrow s $.
\end{enumerate}
\end{definition}
A L\'evy process $L_t$ taking values in $\mathbb{R}^n$ is characterized by a drift vector $b \in {\mathbb{R}^n}$, an $n \times n$ non-negative-definite, symmetric covariance matrix $ Q $ and a Borel measure $\nu$ defined on ${\mathbb{R}^n}\backslash \{ 0\} $.   We call $(b,Q,\nu)$ the generating triplet of  the L\'evy motions $L_t$ . Moreover, we have the L\'evy-It\^o decomposition for $L_t$ as follows:
\begin{equation}
{L_t}= bt + B_{Q}(t) + \int_{|y|< 1} y \widetilde N(t,dy) + \int_{|y|\ge 1} y N(t,dy),
\end{equation}
where $N(dt,dy)$ is the Poisson random measure, $\widetilde N(dt,dy) = N(dt,dy) - \nu(dx)dt$ is the compensated Poisson random measure, $\nu (A) = \mathbb{E}N(1,A)$ is the jump measure, and $ B_{Q}(t)$ is an independent standard $n$-dimensional Brownian motion. The characteristic function of $L_t$ is given by
\begin{equation}
\mathbb{E}[\exp({\rm i}\langle u, L_t \rangle)]=\exp(t\rho(u)), ~~~u \in {\mathbb{R}^n},
\end{equation}
where the function $\rho:{\mathbb{R}^n}\rightarrow \mathbb{C}$ is  the characteristic exponent
\begin{equation}
\rho(u)={\rm i}\langle u, b\rangle-\frac{1}{2}\langle u, Qu \rangle+\int_{{\mathbb{R}^n}\backslash \{ 0\}}{(e^{{\rm i}\langle u, z \rangle}-1-{\rm i}\langle u,z\rangle {I_{\{ | z | \textless 1\} }})\nu(dz)}.
\end{equation}
The Borel measure $\nu$ is called the jump measure. Here $\langle \cdot, \cdot \rangle $ denotes the scalar product in $ \mathbb{R}^n $.
\begin{definition}
For $\alpha \in (0,2)$, a $n$-dimensional symmetric $\alpha $-stable process $ L^{\alpha}_{t} $ is a L\'evy process with characteristic exponent $\rho$
\begin{equation}
\rho(u)=-C_1(n,\alpha)| u |^{\alpha},  ~for~u \in {\mathbb{R}^{n}}
\end{equation}
with $ C_1(n, \alpha):=\pi^{-\frac{1}{2}}\Gamma((1+\alpha)/2)\Gamma(n/2)/\Gamma((n+\alpha)/2)$.
\end{definition}

For a $n$-dimensional symmetric $\alpha$-stable L\'evy process, the diffusion matrix $ Q= 0$,
the drift vector $ b= 0$, and the L\'evy measure $\nu $ is given by
\begin{equation}
\nu(du)=\frac{C_2(n,\alpha)}{{| u |}^{n+\alpha}}du,
\end{equation}
where $ C_2(n, \alpha):=\alpha\Gamma((n+\alpha)/2)/{(2^{1-\alpha}\pi^{n/2}\Gamma(1-\alpha/2))}$.
\subsection{\textbf{Random dynamical systems} (see \cite{la})}

\begin{definition}
Let $ (H,\mathscr{B}(H))$ be a measurable space. A mapping
\begin{equation}
\phi: \mathbb{R}\times \Omega\times H\mapsto H, ~~(t, \omega, x)\mapsto \phi(t,\omega, x)
\end{equation}
with the following properties is called a measurable random dynamical system (RDS), if it is jointly $\mathscr{B}(\mathbb{R})\otimes\mathscr{F}\otimes\mathscr{B}(H)/\mathscr{B}(H)$ measurable and satisfies
the cocycle property:
\begin{equation}
\label{cocycle}
\begin{array}{c}
\phi(0,\omega,\cdot)={\rm id}_{H}, \text{ for each}\ \omega\in\Omega;\\
\phi(t+s,\omega,\cdot)=\phi\big(t,\theta_s\omega,\phi(s,\omega,\cdot)\big),\ \text{ for each}\ s,\,t\in\mathbb{R},\ \omega\in\Omega.
\end{array}
\end{equation}
\end{definition}
Here  we present a simple example to illustrate the abstract definitions of measurable random dynamical system.
\begin{example}
Consider a scalar linear stochastic differential equation
\begin{equation}
dX_t=2dB_t, X_0=x.
\end{equation}
The solution mapping is $\phi(t,\omega,x)=2B_t(\omega)+x$. Note that $\phi(0,\omega,x)=x$, and $\phi(t+s, \omega,x)=2B_{t+s}(\omega)+x=\phi(t,\theta_{s}(\omega), \varphi(s, \omega, x))$.
\end{example}
\subsection{\textbf{Random invariant manifolds} (see \cite{la})} Let $\phi$ be a random dynamical system on the normed space $(H, \mathscr{B}(H))$. We introduce a random invariant manifold with respect to $\phi$.

A family of nonempty closed set $\mathcal{M}=\{\mathcal{M}(\omega)\}_{\omega \in \Omega}$ is called a random set if for every $ z \in H $, the mapping
\[\Omega\ni \omega\mapsto{\rm dist}(z,\mathcal{M}(\omega)):=\inf\limits_{z^{'}\in \mathcal{M}(\omega)}||z-z^{'}||_{H}
\]
is measurable. Moreover, $ \mathcal{M} $ is called a positively invariant set with respect to the random dynamical system $\phi$ if
\begin{equation}\label{invariant set}
\phi(t,\omega,\mathcal{M}(\omega))\subseteq \mathcal{M}(\theta_t\omega), \, \text{ for }\ t\in\mathbb{R^{+}},\ \omega\in\Omega.
\end{equation}

In the sequel, we suppose $ H=\mathbb{R}^{n} \times \mathbb{R}^{m}$ and consider random sets defined by a Lipschitz continuous graph. Let
\begin{equation}
\begin{aligned}
\gamma:\Omega\times\mathbb{R}^{m}\rightarrow\mathbb{R}^{n}\\
(\omega,y)\mapsto \gamma(\omega,y)
\end{aligned}
\end{equation}
be a function  such that $ \gamma(\omega,y)$ is  globally Lipschitzian in $y$ and for  any  $y\in\mathbb{R}^{m}$  and that the mapping $\omega\mapsto \gamma(\omega,y)$ is a random variable. Then $\mathcal{M}(\omega):=\{(\gamma(\omega,y),y):y\in\mathbb{R}^{m}\}$ is a random set (see \cite{BS}). The invariant random set $\mathcal{M}(\omega)$ is called a Lipschitz random invariant manifold.

\section{Existence of a random invariant manifold}
In this section, we will present the reduction method for a slow-fast stochastic dynamical system  with  $\alpha$-stable L\'evy process. Here we only present   main results which will be used in the next section. The detailed illustrations will be presented in the Appendix (Section 6).

Let $\mathcal{D}_{0}(\mathbb{R},\mathbb{R}^{n})$ be the space of c\`{a}dl\`{a}g functions starting at $0$ given by
\begin{equation*}
  \mathcal{D}_0=\{\omega :\ \text{for } \forall t\in\mathbb{R},\ \lim_{s\uparrow t}\omega(s)=\omega(t-),\  \lim_{s\downarrow t}\omega(s)=\omega(t)\ \text{exist}\ \text{and} \ \omega(0)=0  \}.
\end{equation*}
For functions $\omega_{1}, \omega_{2} \in \mathcal{D}_{0}$, ${\rm d}_{\mathbb{R}}(\omega_{1}, \omega_{2})$ is given by
\begin{eqnarray*}
  {\rm d}_{\mathbb{R}}(\omega_{1},\omega_{2})=\inf\left\{\varepsilon>0:
     |\omega_{1}(t)-\omega_{2}(\lambda t)|\leq\varepsilon,\ \big|\ln\frac{\arctan(\lambda t)-\arctan(\lambda s)}{\arctan(t)-\arctan(s)}\big|\leq\varepsilon\right.
     \\
    \left.\text{for every}~\ t, s\in\mathbb{R}\ \text{and some}\ \lambda\in\mathbb{R}^{\mathbb{R}}\right\},
\end{eqnarray*}
where
\begin{equation*}
\mathbb{R}^{\mathbb{R}}=\{\lambda :\mathbb{R}\rightarrow\mathbb{R};\ \lambda\ \text{is injective increasing},\ \lim\limits_{t\rightarrow-\infty}\lambda(t)=-\infty,\ \lim\limits_{t\rightarrow\infty}\lambda(t)=\infty \}.
\end{equation*}

Denote by $\mathscr{B}(\mathcal{D}_{0})$   the associated Borel $\sigma$-algebra generated by $ \mathcal{D}_{0}$, then $(\mathcal{D}_{0},\mathscr{B}(\mathcal{D}_{0}))$ is a separable metric space.
The probability measure $\mathbb{P}$ is generated by $\mathbb{P}(\mathcal{D}_{0}\cap A):=\mathbb{P}_{\mathbb{R}}(A)$ for each $A\in\mathscr{B}({\mathbb{R}})$. The Wiener shift
\begin{equation}
\begin{aligned}
  &\theta:\mathbb{R}\times\mathcal{D}_{0}\rightarrow\mathcal{D}_{0},\\
  &\theta_{t}\omega(\cdot)\mapsto\omega(\cdot+t)-\omega(t).
\end{aligned}
\end{equation}
is a Carath\'eodory function. Obviously, the Wiener shift $\theta$ is jointly measurable. The probability measure $ \mathbb{P}$ is $ \theta$-
invariant and the metric dynamical system
\begin{equation}
(\mathcal{D}_{0},\mathscr{B}(\mathcal{D}_{0}),\mathbb{P},(\theta_{t})_{t\in\mathbb{R}})
\end{equation}
is ergodic (see \cite{Kai}).

Define $(\Omega^1,\mathcal{F}^1,\mathbb{P}^1,\theta_t^1)=(\mathcal{D}_{0},\mathscr{B}(\mathcal{D}_{0}),\mathbb{P},(\theta_{t})_{t\in\mathbb{R}})$.
Then $(\Omega^1,\mathcal{F}^1,\mathbb{P}^1,\theta_t^1)$ is a metric dynamical system. Similarly, we define
$ \Omega^2, \mathcal{F}^2,\mathbb{P}^2,\theta_t^2 $. Then $ (\Omega^2, \mathcal{F}^2,\mathbb{P}^2,\theta_t^2) $ is another metric dynamical system.  Introduce
\begin{equation}
\Omega:=\Omega^1\times\Omega^2, ~\mathcal{F}:=\mathcal{F}^1\times\mathcal{F}^2, ~\mathbb{P}:=\mathbb{P}^1\times\mathbb{P}^2,~
\theta_t:=\theta_t^1\times\theta_t^2,
\end{equation}
and then $(\Omega, \mathcal{F}, \mathbb{P}, \theta_t)$ is a metric dynamical system.

In this paper, we consider the following fast-slow stochastic dynamical system
\begin{equation}
\label{rde110}
\left\{
\begin{aligned}
d{X^\varepsilon _t} &=\frac{1}{\varepsilon}AX^\varepsilon_t dt+ \frac{1}{\varepsilon}{F}({X^\varepsilon_t},{Y^\varepsilon_t})dt + \frac{\sigma_1}{\varepsilon ^{\frac{1}{\alpha _1}}}dL_1^{{\alpha _1}},  \;\; X^\varepsilon _0=x \in \mathbb{R}^n, \\
d{Y^\varepsilon _t} &= BY^\varepsilon_t dt+{G}({X^\varepsilon_t},{Y^\varepsilon _t})dt +\sigma_2 dL_2^{{\alpha _2}},  \;\; \;\;   Y^\varepsilon _0=y\in \mathbb{R}^m.
\end{aligned}
\right.
\end{equation}
Here $({X^\varepsilon_t},{Y^\varepsilon_t})$ is an ${\mathbb{R}^n} \times {\mathbb{R}^m}$-valued signal process which contains the fast and slow components, respectively. The interaction functions $ F: {\mathbb{R}^n} \times {\mathbb{R}^m}\rightarrow {\mathbb{R}^n}$ and $G: {\mathbb{R}^n} \times {\mathbb{R}^m}\rightarrow {\mathbb{R}^m}$ are Borel measurable respectively.  The constant matrixes $A$ and $B$ are ${\mathbb{R}^n} \times {\mathbb{R}^n}$ and ${\mathbb{R}^m} \times {\mathbb{R}^m}$ respectively. Both $\sigma_1$ and $ \sigma_2$ are nonzero real noise intensities. The non-Gaussian  processes  $L_1^{{\alpha _1}},L_2^{{\alpha _2}}$  (with $1<\alpha_1, \alpha_2 <2$) are independent symmetric $\alpha$-stable L\'evy processes with triplets $(0,0,{\nu_1})$ and $(0,0,{\nu _2})$, respectively.   The small parameter $\varepsilon$ $(0<\varepsilon \ll 1)$ is the ratio of the two time scales.
We make the following assumptions on the signal system.

\par
{\bf Hypothesis  H.1}
There exists a constant $ M_{A}>0 $ such that
\begin{equation*}
(Ax,x)\leq -M_{A}|x|^2,~~\mbox{for~all}~x\in\mathbb{R}^{n}.
\end{equation*}
For every $y\in\mathbb{R}^m$, $|By|\leq \|B\||y|$, where $\|B\|$  standing for the norm of the matrix $B$, and $-B$ has no eigenvalue on the imaginary axis.

\par
{\bf Hypothesis  H.2}
The functions $ F, ~G $ satisfy the global Lipschitz conditions and sublinear growth condition, i.e., there exists positive constants $L$ and $K$ such that for all $x_i\in \mathbb{R}^{n},y_i\in \mathbb{R}^{m}, \; i=1, 2$ or $x\in \mathbb{R}^{n}, y\in \mathbb{R}^{m}$, we have
\begin{equation*}
\begin{aligned}
| F(x_1,y_1)-F(x_2,y_2)|\leq L (| x_1-x_2|+| y_1-y_2|), \\
| G(x_1,y_1)-G(x_2,y_2)| \leq L (| x_1-x_2|+| y_1-y_2|),
\end{aligned}
\end{equation*}
and
\begin{equation*}
|F(x,y)|^{2}+| G(x,y|^{2} \leq K \left[ 1+|x|^2+|y|^2 \right].
\end{equation*}
\par
{\bf Hypothesis  H.3 } $ M_{A} \textgreater L$.

\par
{\bf Hypothesis  H.4 }
The function $F$ and $G$ are uniformly bounded, i.e.,
\begin{equation}
\begin{aligned}
\sup_{( x,y)\in \mathbb{R}^{n}\times \mathbb{R}^{m}}|F( x,y)|&<\infty, \\
\sup_{( x,y)\in \mathbb{R}^{n}\times \mathbb{R}^{m}}|G( x,y)|&<\infty.
\end{aligned}
\end{equation}
\begin{remark}
Hypothesis { \bf{H.3}} can be interpreted as a spectral gap condition.
\end{remark}
\begin{remark}
 Under the  assumptions ${\bf{H.2}}$ and ${ \bf{H.4}}$, the system (\ref{rde110}) has a global unique
solution   $(X^\varepsilon(t), Y^\varepsilon(t))$,  with a given  initial value $(x(0), y(0))$.

\end{remark}
Introduce the following two auxiliary systems
\begin{equation}
\label{ori6}
\begin{aligned}
d\xi^{\varepsilon}(t)&=\frac{1}{\varepsilon} A\xi^{\varepsilon}(t) dt+\frac{\sigma_1}{\varepsilon^{\frac{1}{\alpha_1}}}dL_t^{\alpha_1}(\omega_1),~~\xi(0)=\xi_{0}, \\
d\eta(t) &=B\eta(t) dt+\sigma_2dL_t^{\alpha_2}(\omega_2), ~~ \eta(0)=\eta_0.
\end{aligned}
\end{equation}
By [Appendix, Lemma 1], there exist two random variables $\xi^{\varepsilon}(t)$ and $\eta(t)$ such that  $\xi^{1, \frac{1}{\varepsilon}}(\theta_t^1\omega_1)$ and $\eta(\theta_t^2\omega_2)$ solve the equations \eqref{ori6}. Set
\begin{equation}\label{T}
\left(
  \begin{array}{ccc}
 \bar{ x}^{\varepsilon}(t)\\
  \bar{y}^{\varepsilon}(t)\\
  \end{array}
\right):=
  \left(
  \begin{array}{ccc}
    &X^\varepsilon(t)-\xi^{1, \frac{1}{\varepsilon}}(\theta_t^1\omega_1)\\
    &Y^\varepsilon(t)-\eta(\theta_t^2 \omega_2)\\
  \end{array}
\right),
\end{equation}
then $(\bar{x}^\varepsilon(t), \bar{y}^\varepsilon(t))$ satisfies the following equations
\begin{eqnarray}
\label{rdem}
\left\{\begin{array}{l}
\dot{\bar{x}}^\e(t)=\frac{1}{\e}A{\bar{x}}^\e(t)+\frac{1}{\e}F\big({\bar{x}^\e}(t)+\xi^{1, \frac{1}{\varepsilon}}(\theta^1_t\omega_1),{\bar{y}^\e}(t)+\eta(\theta^2_t{\omega_2})\big),\\
\dot{\bar{y}}^{\varepsilon}(t)=B{\bar{y}^\e}(t)+G\big({\bar{x}^\e}(t)+\xi^{1, \frac{1}{\varepsilon}}(\theta^1_t\omega_1),{\bar{y}^\e}(t)+\eta(\theta^2_t{\omega_2})\big).

\end{array}
\right.
\end{eqnarray}
The scaling $ t \rightarrow \varepsilon t $ in \eqref{rdem} yields
\begin{equation}
\label{rrde}
\left\{
\begin{aligned}
\frac{dx}{dt}&= Ax+ F\big(x+\xi^{1, \frac{1}{\varepsilon}}(\theta^1_{\e t}\omega_1),y+\eta(\theta^2_{\e t}{\omega_2})\big),\\
\frac{dy}{dt}&= \varepsilon By+\varepsilon G\big(x+\xi^{1, \frac{1}{\varepsilon}}(\theta^1_{\e t}\omega_1),y+\eta(\theta^2_{\e t}{\omega_2})\big).\\
\end{aligned}
\right.
\end{equation}

If we now replace $\xi^{1,\frac{1}{\varepsilon}}(\theta^1_{\e t}\omega_1)$ by $ \xi^{1, 1}(\theta^1_{ t}\omega_1)$  that has the same
distribution, then we obtain a system of the following form
\begin{equation}
\label{rrde1}
\left\{
\begin{aligned}
\frac{d\bar{X}^{\varepsilon}(t)}{dt}&= A\bar{X}^{\varepsilon}(t)+ F\big(\bar{X}^{\varepsilon}(t)+\xi^{1, 1}(\theta^1_{ t}\omega_1),\bar{Y}^{\varepsilon}(t)+\eta(\theta^2_{\e t}{\omega_2})\big),\\
\frac{d\bar{Y}^{\varepsilon}(t)}{dt}&= \varepsilon B\bar{Y}^{\varepsilon}(t)+\varepsilon G \big(\bar{X}^{\varepsilon}(t)+\xi^{1, 1}(\theta^1_{ t}\omega_1),\bar{Y}^{\varepsilon}(t)+\eta(\theta^2_{\e t}{\omega_2})\big).\\
\end{aligned}
\right.
\end{equation}
The system \eqref{rrde1} has a unique global solution $\phi^{\varepsilon}(t, \omega, (x_0,y_0))$ for any initial condition $(x_0,y_0)\in \mathbb{R}^{n+m}$, the solution operator of the initial value problem to system \eqref{rrde1} denoted by
\begin{equation}
\phi^{\varepsilon}(t,\omega, (x_0,y_0)):=\big(\phi^{\varepsilon}_1\big(t,\omega,(x_0,y_0)\big), \phi^{\varepsilon}_2\big(t,\omega,(x_0,y_0)\big)\big)
\end{equation}
Let
\begin{equation*}
\left\{
\begin{aligned}
F(\theta^{\e}_{ t}{\omega},\bar{X}^{\varepsilon},\bar{Y}^{\varepsilon})&=F\big(\bar{X}^{\varepsilon}(t)+\xi^{1, 1}(\theta^1_{ t}\omega_1),\bar{Y}^{\varepsilon}(t)+\eta(\theta^2_{\e t}{\omega_2})\big)\\
G(\theta^{\e}_{ t}{\omega},\bar{X}^{\varepsilon},\bar{Y}^{\varepsilon})&=G \big(\bar{X}^{\varepsilon}(t)+\xi^{1, 1}(\theta^1_{ t}\omega_1),\bar{Y}^{\varepsilon}(t)+\eta(\theta^2_{\e t}{\omega_2})\big)\\
\end{aligned}
\right.
\end{equation*}
Then the equation \eqref{rrde1} can be rewritten as
\begin{equation}
\label{rrde4}
\left\{
\begin{aligned}
\frac{d\bar{X}^{\varepsilon}(t)}{dt}&=  A\bar{X}^{\varepsilon}(t)
+ F(\theta^{\e}_{ t}{\omega},\bar{X}^{\varepsilon},\bar{Y}^{\varepsilon})\\
\frac{d\bar{Y}^{\varepsilon}(t)}{dt}&=\varepsilon B\bar{Y}^{\varepsilon}(t)+\varepsilon G(\theta^{\e}_{ t}{\omega}, \bar{X}^{\varepsilon},\bar{Y}^{\varepsilon}).
\end{aligned}
\right.
\end{equation}
The following theorem indicates the random dynamical system defined by \eqref{rrde4} has a random invariant manifold.

\begin{theorem}
Assue that the hypotheses $(\mathbf{H.1})$-$(\mathbf{H.4})$ hold.  Then for sufficiently small $ \varepsilon $ and sufficiently large $ T $, the random dynamical system defined by \eqref{rrde4} has a random invariant manifold,  with graph   defined by $ \gamma^{\e}(\omega)$.
\end{theorem}
\begin{proof}
Assume $\mathcal{A}^{\varepsilon}$ has a unique fixed point $ \gamma^{\varepsilon}$, where $\mathcal{A}^{\varepsilon}$ is defined in the appendix \eqref{define}. Hence
replacing $ \omega $ by $ \theta^{\varepsilon}_{T}\omega$, we have
\begin{equation}
\psi^{\varepsilon}(t,\omega,\gamma^{\varepsilon})=\gamma^{\varepsilon}(\theta^{\varepsilon}_{T}\omega, \cdot).
\end{equation}
Therefore for $ t > 0 $, we have
\begin{equation}
\begin{aligned}
\psi^{\varepsilon}\big(t,\omega,\gamma^{\varepsilon}(\omega)\big)&=\psi^{\varepsilon}\big(t,\cdot,\mathcal{A}^{\varepsilon}
\big(\gamma^{\varepsilon}(\cdot)\big)\big)(\omega)\\
&=\psi^{\varepsilon}\big(t,\varepsilon, \psi^{\varepsilon}(T, \theta^{\varepsilon}_{-T}\omega,
\gamma^{\varepsilon}(\theta^{\varepsilon}_{-T}\omega))\big)\\
&=\psi^{\varepsilon}(t+T,\theta^{\varepsilon}_{-T}\omega, \gamma^{\varepsilon}(\theta^{\varepsilon}_{-T}\omega))\\
&=\psi^{\varepsilon}\big(T,\theta^{\varepsilon}_{-T+t}\omega, \psi^{\varepsilon}\big(T, \theta^{\varepsilon}_{-T}\omega,\gamma^{\varepsilon}(\theta^{\varepsilon}_{-T}\omega)\big)\big)\\
&=\mathcal{A}^{\varepsilon}\big(\psi^{\varepsilon}(t, \theta^{\varepsilon}_{-t}\cdot,\gamma^{\varepsilon}(\theta^{\varepsilon}_{-t}\cdot )\big)(\theta^{\varepsilon}_t\omega) \\
&=\gamma^{\varepsilon}(\theta^{\varepsilon}_t).
\end{aligned}
\end{equation}
Therefore the random dynamical system \eqref{rrde4} has a random invariant Lipschitz manifold defined by $\big\{\big(\gamma^{\varepsilon}(\omega, \tilde{y}),\tilde{y}\big)\big\}$. 
\end{proof}
\begin{remark}
Thus the systems derived from \eqref{rdem} has a random invariant manifold with the same graph.
\end{remark}

Based on the relationship between \eqref{rde110}and \eqref{rdem}, it holds that the system \eqref{rde110} has a random invariant manifold
\begin{equation}
M^{\varepsilon}(\omega)=\{\gamma^{\varepsilon}(\omega, \widetilde{y})+\xi^{1, \frac{1}{\varepsilon}}(\omega_1),\widetilde{y}+\eta(\omega_2)\}.
\end{equation}
Therefore we can get the following reduction system which describes the behavior for system \eqref{rde110}.
\begin{theorem}
Assume  that the assumptions $(\mathbf{H.1})$-$(\mathbf{H.4})$ hold. Then for any solution $ z^{\epsilon}(t)= (X^{\epsilon}_t, Y^{\epsilon}_t)$ to system \eqref{rde110} with initial data $ z^{\epsilon}(0)= (X^{\epsilon}_0, Y^{\epsilon}_0)$, there exists the following reduced low dimensional systems on the random invariant manifold
\begin{equation}
\label{redda}
\left\{
\begin{aligned}
d{y^\varepsilon _t} &= By^\varepsilon_t dt+{G}({x^\varepsilon_t},{y^\varepsilon _t})dt +\sigma_2 dL_2^{{\alpha _2}},\\
{x^\varepsilon _t} &=\gamma^{\varepsilon}(\theta^{\varepsilon}_t \omega, {y^\varepsilon _t}-\eta(\theta^{2}_t\omega_2))+\xi^{1, \frac{1}{\varepsilon}}(\theta^1_t\omega_1),
\end{aligned}
\right.
\end{equation}
such that for sufficiently  small $ \e$, sufficiently large $ t$ and some positive constant $\tilde{\kappa}$, we have
\begin{equation}
||z^{\varepsilon}(t,\omega)-\widetilde{z}^{\varepsilon}(t,\omega)||_{\infty} \leq C_{\e,\tilde{\kappa}}e^{-\frac{\tilde{\kappa}}{\varepsilon} t}||z^{\varepsilon}(0,\omega)-\widetilde{z}^{\varepsilon}(0,\omega)||_{\infty},
\end{equation}
where $ \widetilde{z}^{\epsilon}(t)= (x^{\epsilon}_t, y^{\epsilon}_t)$  is the solution of the low dimensional system \eqref{redda} with initial data $ \widetilde{z}^{\epsilon}(0)= (x^{\epsilon}_0, y^{\epsilon}_0)$.
\begin{proof}
Lemma 8 from appendix  tells us that the system \eqref{rdem} has an exponentially tracking manifold, then so has the  system \eqref{rde110}.
\end{proof}
\end{theorem}

\section{Approximation analysis for nonlinear filter}
In this  section, we study the nonlinear filtering problem on the random invariant manifold. For $T>0$, an observation system  is given by
\begin{equation}
dZ^{\e}_t=  h(Y^\e_t) dt+dW_t , \quad t\in[0,T],
\end{equation}
where $W_t$ is a standard Brownian motion independent of $L^{\alpha_1}_t$ and $L^{\alpha_2}_t$. In practical applications, we can only observe the slow component $Y^\e_t$. Let $(\Omega, \mathscr{F},\mathbb{P} )$ be a probability space together with a filtration $(\mathscr{F}_{t})_{t\geq 0}$ which satisfies the usual conditions. For the observation system $Z^{\e}_t$, we make the following additional hypothesis.

\par
{\bf Hypothesis  H.5 }
 The sensor function $h$ is bounded and Lipschitz continuous in $x$, with Lipschitz constant   denoted by $\|h\|_{Lip}$.

Let
\begin{equation}
\mathcal{Z}_t=\sigma(Z^\e_s:0\leq s\leq t)\vee \mathcal{N},
\end{equation}
where $\mathcal{N}$ is the  collection of all $\mathbb{P}$ -negligible sets of $(\Omega,\mathscr{F})$.

By the  Girsanov's change of measure theorem, we obtain a new probability measure $\widetilde{\mathbb{P}}$, such that the
observation $Z^{\e}_t$ becomes $\widetilde {\mathbb{P}}$-independent of the signal variables $(X^{\varepsilon}_t,Y^{\varepsilon}_t)$. In fact, this can be done through
\begin{equation}
\left.\frac{{d\widetilde {\mathbb{P}}}}{{d\mathbb{P}}}\right|_{\mathscr{F}_t}= \exp\left(-\sum\limits_{i = 1}^m {\int_0^t {{h^i}({Y^\e_s})dW_s^i} }- \frac{1}{2}\sum\limits_{i = 1}^m {\int_0^t {{h^i}{{({Y^\e_s})}^2}} } ds\right)=:\left(\mathbb{R}^{\e}_t\right)^{-1}.
\end{equation}
Hence  we have
\begin{equation}
\mathbb{R}^{\e}_t=\left.\frac{{d\mathbb{P}}}{{d\widetilde{\mathbb{P}}}}\right|_{\mathscr{F}_t}=\exp\left(\sum\limits_{i = 1}^m {\int_0^t {{h^i}({Y^\e_s})d(Z_s^{\e})^{i}} }- \frac{1}{2}\sum\limits_{i = 1}^m {\int_0^t {{h^i}{{({Y^\e_s})}^2}} } ds\right).
\end{equation}
Define
\begin{equation}
\begin{aligned}
\rho^{\e}_t(\phi)&:=\widetilde{\mathbb{E}}\big[\phi(Y^{\varepsilon}_t)\mathbb{R}^{\e}_t|\mathcal{Z}_t\big],\\
\pi^{\e}_t(\phi)&:=\mathbb{E}\big[\phi(Y^{\varepsilon}_t)|\mathcal{Z}_t\big],
\end{aligned}
\end{equation}
where $\widetilde {\mathbb{E}}$ stands for the expectation under $\widetilde{\mathbb{P}}$ and $\phi \in \mathcal{C}_b^1(\mathbb{R}^m)$ .

Then by the Kallianpur-Striebel formula, we have
\begin{equation}
 \pi^{\e}_t(\phi)=\frac{\rho^{\e}_t(\phi)}{\rho^{\e}_t(1)}.
\end{equation}
Here $\pi^{\e}_t(\phi)$ is called the normalized filtering  of $Y^{\varepsilon}_{t}$ with respect to $\mathcal{Z}_t$.

On the other hand,  we can rewrite the reduced system (\ref{redda}) as
\begin{equation}
\label{redd12}
d{\widetilde{Y}}^\e_t=B\widetilde{Y}^\e_tdt+\widetilde{G}^\e(\omega, \widetilde{Y}^\e_t)dt+\sigma_2d{L_2^{{\alpha _2}}}
\end{equation}
with
\begin{equation}
\widetilde{G}^\e(\omega, \widetilde{Y}^\e_t):=G\big(\gamma^{\varepsilon}\big(\theta^{\varepsilon}_t \omega, {\widetilde{Y}^\varepsilon _t}-\eta(\theta^{2}_t\omega_2)\big)+\xi^{1, \frac{1}{\varepsilon}}(\theta^1_t\omega_1),\widetilde{Y}^\e_t\big).
\end{equation}
In the following, we are more interested in the reduced filtering problem with the actual observation, i.e., we will study the nonlinear filtering problem for the reduced system \eqref{redd12} with the actual observation $Z^{\e}_t$. Set
\begin{equation}
\widetilde{\mathbb{R}}^\e_t:=\exp\left\{\int_0^t h(\widetilde{Y}^\e_s)dZ^{\e}_s-\frac12\int_0^t|h(\widetilde{Y}^\e_s)|^2 ds\right\},
\end{equation}
then $\widetilde{\mathbb{R}}^\e_t$ is an exponential martingale under $\widetilde {\mathbb{P}}$. Thus, we can define the
``formally'' non-normalized filtering for $\widetilde{Y}^\e_t$ by
\begin{equation}
\widetilde{\rho}_t^\e(\phi) :=\widetilde{\mathbb{E}}[\phi(\widetilde{Y}^\e_t)\widetilde{\mathbb{R}}^\e_t|\mathcal{Z}_t].
\end{equation}
And set
\begin{equation}
\widetilde{\pi}_t^\e(\phi):=\frac{\widetilde{\rho}_t^\e(\phi)}{\widetilde{\rho}_t^\e(1)},
\end{equation}
then $\widetilde{\pi}_t^\e$ could be understood as the nonlinear filtering problem for $\widetilde{Y}^\e_t$ with respect to $\mathcal{Z}_t$.
%

Now, we are ready to state and prove the main result in this paper.
\begin{theorem}
Assume the hypotheses  $(\bf{H.1})$--$(\bf{H.5})$ hold.  Then for $ \forall \min(\frac{1}{4}\alpha_1,\frac{1}{4}\alpha_2)\textgreater   ~p  \textgreater 0 $,  $\e$ sufficiently small and $t\in[0,T]$,  there exists a positive constant $C$ such that for
$\phi\in \mathcal{C}_b^1(\mathbb{R}^m)$
\begin{equation}
\mathbb{E}|\pi^{\e}_t(\phi)- \widetilde{\pi}_t^\e(\phi)|^p
\leq C ||\phi||^{p}\big(\mathbb{E}||z^{\e}(0,\omega)-\widetilde{z}^{\e}(0,\omega)||^{16p}_{\infty}\big)^{\frac{1}{16}}\left(e^{\frac{-4\tilde{\kappa} tp}{\varepsilon}}+\varepsilon \right)^{\frac{1}{4}}.
\end{equation}
\end{theorem}
\begin{proof}
Step 1.  As in  \cite{HJ}, we have
\begin{equation}
\mathbb{E}\left|\widetilde{\rho}^{\e}_t(1)\right|^{-p} \leq \exp\left\{(2p^2+p+1)CT/2\right\}, \quad t\in[0,T].
\end{equation}

Step 2. For $\phi\in \mathcal{C}^1_b(\mathbb{R}^m)$, it follows from the H\"older inequality that
\begin{equation}
\begin{aligned}
\mathbb{E}\left|\rho^{\e}_t(\phi)-\widetilde{\rho}^{\e}_t(\phi)\right|^p &=\widetilde{\mathbb{E}}\left[\left|\rho^{\e}_t(\phi)-\widetilde{\rho}^{\e}_t(\phi)\right|^p
\mathbb{R}^\e_t\right] \\
&\leq \left(\widetilde{\mathbb{E}}\left[\left|\rho^{\e}_t(\phi)-\widetilde{\rho}^{\e}_t(\phi)\right|^{2p}\right]\right)^{\frac{1}{2}}
\big(\widetilde{\mathbb{E}}\big[\mathbb{R}^\e_t\big]^{2}\big)^{\frac{1}{2}} \\
&\leq \exp(CT/2)\left(\widetilde{\mathbb{E}}\left[\left|\rho^{\e}_t(\phi)-\widetilde{\rho}^{\e}_t(\phi)\right|^{2p}\right]\right)^{\frac{1}{2}}.
\end{aligned}
\end{equation}

Step 3. Using the definitions of $\rho^{\e}_t(\phi)$ and $\widetilde{\rho}^{\e}_t(\phi)$, Jensen's inequality and H\"older inequality, we get
\begin{equation}
\label{i1i2}
\begin{aligned}
\widetilde{\mathbb{E}}\left|\rho^{\e}_t(\phi)-\widetilde{\rho}^{\e}_t(\phi)\right|^{2p}&=
\widetilde{\mathbb{E}}\left|\widetilde{\mathbb{E}}[\phi(Y_t^\e)\mathbb{R}^\e_t|\mathcal{Z}_t]-\widetilde{\mathbb{E}}[\phi(\widetilde{Y}^\e_t)
\widetilde{\mathbb{R}}^\e_t|\mathcal{Z}_t]\right|^{2p}\\
&\leq2^{2p-1}\widetilde{\mathbb{E}}\left[\left|\phi(Y_t^\e)\mathbb{R}^\e_t-\phi(\widetilde{Y}^\e_t)\mathbb{R}^\e_t\right|^{2p}\right]
+2^{2p-1}\widetilde{\mathbb{E}}\left[\left|\phi(\widetilde{Y}_t^\e)\mathbb{R}^\e_t-\phi(\widetilde{Y}^\e_t)
\widetilde{\mathbb{R}}^\e_t\right|^{2p}\right]\\
&\leq2^{2p-1}||\phi||^{2p}e^{p(4p-1)CT}\left(\widetilde{\mathbb{E}}\left|Y_t^\e-\widetilde{Y}^\e_t\right|^{4p}\right)^{\frac{1}{2}}+ 2^{2p-1}\|\phi\|^{2p}\widetilde{\mathbb{E}}\left[\left|\mathbb{R}^\e_t-\widetilde{\mathbb{R}}^\e_t\right|^{2p}\right]\\
\end{aligned}
\end{equation}

Step 4. By the It\^o's formula, BDG inequality and bounded property of $h$, we have
\begin{equation}
\begin{aligned}
\widetilde{\mathbb{E}}\left[\left|\mathbb{R}^\e_t-\widetilde{\mathbb{R}}^\e_t\right|^{2p}\right]&=\widetilde{\mathbb{E}}\left[\left|\int_0^t\left(\mathbb{R}^\e_s
h(Y_s^\e)-\widetilde{\mathbb{R}}^\e_s h(\widetilde{Y}^\e_s)\right)d
Z_s^\e\right|^{2p}\right]\\
&\leq C\widetilde{\mathbb{E}}\left[\int_0^t\left|\mathbb{R}^\e_s
h(Y_s^\e)-\widetilde{\mathbb{R}}^\e_s h(\widetilde{Y}^\e_s)\right|^{2}ds\right]^{p}\\
&\leq 2^{2p-1}CT^{p-1}\int_0^t\widetilde{\mathbb{E}}\left|\mathbb{R}^\e_s
h(Y_s^\e)-{\mathbb{R}}^\e_s h(\widetilde{Y}^\e_s)\right|^{2p}ds\\
&+2^{2p-1}CT^{p-1}\int_0^t\widetilde{\mathbb{E}}\left|\mathbb{R}^\e_s
h(\widetilde{Y}_s^\e)-\widetilde{\mathbb{R}}^\e_s h(\widetilde{Y}^\e_s)\right|^{2p}ds\\
\end{aligned}
\end{equation}

Then we have
$$
\widetilde{\mathbb{E}}\left[\left|\mathbb{R}^\e_t-\widetilde{\mathbb{R}}^\e_t\right|^{2p}\right]\leq
C\big(\widetilde{\mathbb{E}}||z^\e(0)-\widetilde{z}^\e(0)||^{4p}_{\infty}\big)^{\frac{1}{2}}+C\int_0^t\widetilde{\mathbb{E}}\left|\mathbb{R}^\e_s-\widetilde{\mathbb{R}}^\e_s
\right|^{2p}d s.
$$
Using Gr\"onwall's inequality, we obtain that
\begin{equation}
\label{rr}
\widetilde{\mathbb{E}}\left[\left|\mathbb{R}^\e_t-\widetilde{\mathbb{R}}^\e_t\right|^{2p}\right]\leq C\big(\widetilde{\mathbb{E}}||z^\e(0)-\widetilde{z}^\e(0)||^{4p}_{\infty}\big)^{\frac{1}{2}}.
\end{equation}

Step 5. Combing the results from Step 3 and Step 4, we get
\begin{equation}
\mathbb{E}\left|\rho^{\e}_t(\phi)-\widetilde{\rho}^{\e}_t(\phi)\right|^p
\leq C||\phi||^{p}\big(\widetilde{\mathbb{E}}||z^\e(0)-\widetilde{z}^\e(0)||^{4p}_{\infty}\big)^{\frac{1}{4}}\left(e^{\frac{-2\tilde{\kappa} tp}{\varepsilon}} +\varepsilon\right)^{\frac{1}{2}}.
\end{equation}

Step 6. Using the relationships between $\pi^{\e}_t(\phi), ~\widetilde{\pi}_t^\e(\phi) $ and $ \rho^{\e}_t(\phi), \widetilde{\rho}^{\e}_t(\phi)$, we have
\begin{equation}
\begin{aligned}
\mathbb{E}|\pi^{\e}_t(\phi)- \widetilde{\pi}_t^\e(\phi)|^p &=\mE\left|\frac{\rho^{\e}_t(\phi)-\tilde{\rho}^{\e}_t(\phi)}{\tilde{\rho}^{\e}_t(1)}-\pi^{\e}_t(\phi)\frac{\rho^{\e}_t(1)-\tilde{\rho}^{\e}_t(1)}{\tilde{\rho}^{\e}_t(1)}\right|^{p}\\
&\leq2^{p-1}\mE\left|\frac{\rho^{\e}_t(\phi)-\tilde{\rho}^{\e}_t(\phi)}{\tilde{\rho}^{\e}_t(1)}\right|^{p}+2^{p-1}\mE\left|\pi^{\e}_t(\phi)\frac{\rho^{\e}_t(1)-\tilde{\rho}^{\e}_t(1)}{\tilde{\rho}^{\e}_t(1)}\right|^{p}\\
&\leq2^{p-1}\left(\mE\left|\rho^{\e}_t(\phi)-\tilde{\rho}^{\e}_t(\phi)\right|^{2p}\right)^{1/2}\left(\mE\left|\tilde{\rho}^{\e}_t(1)\right|^{-2p}\right)^{1/2}\\
&+2^{p-1}\|\phi\|^{p}\left(\mE\left|\rho^{\e}_t(1)-\tilde{\rho}^{\e}_t(1)\right|^{2p}\right)^{1/2}\left(\mE\left|\tilde{\rho}^{\e}_t(1)\right|^{-2p}\right)^{1/2}\\
&\leq  C ||\phi||^{p}\big(\mathbb{E}||z^{\e}(0,\omega)-\widetilde{z}^{\e}(0,\omega)||^{16p}_{\infty}\big)^{\frac{1}{16}}\left(e^{\frac{-4\tilde{\kappa} tp}{\varepsilon}}+\varepsilon \right)^{\frac{1}{4}}.
\end{aligned}
\end{equation}

\end{proof}
\begin{remark}
For a stochastic differential equation   with $\alpha$-stable L\'evy noise, the existence for its $p$ moment requires $1<p<\alpha$.
\end{remark}

\section{Conclusion}

We have presented a theoretical foundation for the development of effective filtering on a  random invariant manifold in complex non-Gaussian multiscale systems.   We can  further  extend this work to the case when the observation system is also driven by non-Gaussian noise. i.e.
\begin{equation}
dZ^{\e}_t=  h(Y^\e_t) dt+dL_t , \quad t\in[0,T],
\end{equation}
where $ L_t $ is a standard L\'evy process independent of $L_t^{\alpha_1}$ and $L_t^{\alpha_2}$.

Moreover,   we can extend this work  to slow-fast non-Gaussian stochastic \emph{partial} differential equations,  i.e.,  we consider the following fast-slow stochastic dynamical system
\begin{equation*}
\label{rde11}
\left\{
\begin{aligned}
d{X^\varepsilon _t} &=\frac{1}{\varepsilon}AX^\varepsilon_t dt+ \frac{1}{\varepsilon}{F}({X^\varepsilon_t},{Y^\varepsilon_t})dt + \frac{\sigma_1}{\varepsilon ^{\frac{1}{\alpha _1}}}dL_1^{{\alpha _1}},   \\
d{Y^\varepsilon _t} &= BY^\varepsilon_t dt+{G}({X^\varepsilon_t},{Y^\varepsilon _t})dt +\sigma_2 dL_2^{{\alpha _2}}.
\end{aligned}
\right.
\end{equation*}
Here $({X^\varepsilon_t},{Y^\varepsilon_t})$ is defined in two separable Hilbert space ${\mathbb{H}} \times {\mathbb{H}}$ -valued signal process which represents the fast and slow components. The interaction functions $ F: {\mathbb{H}} \times {\mathbb{H}}\rightarrow {\mathbb{H}}, G: {\mathbb{H}} \times {\mathbb{H}}\rightarrow {\mathbb{H}}$ are Borel measurable respectively. $A:\mathcal{D}(A)\subset H\mapsto H$ is a self-adjoint compact operator on $\mathbb{H}$ such that $-A$ has discrete spectrum $0<\mu_1<\mu_2<\cdots \mu_{k}<\cdots$ and $\lim_{k\rightarrow \infty}\mu_{k}=\infty$. $B:\mathcal{D}(B)\subset H\mapsto H$ is a linear unbounded operator on $\mathbb{H}$ such that $-B$ has discrete spectrum $0<\lambda_1<\lambda_2<\cdots \lambda_{k}<\cdots$ and $\lim_{k\rightarrow \infty}\lambda_{k}=\infty$. Both $\sigma_1$ and $ \sigma_2$ are nonzero real noise intensities.  The parameter $\varepsilon$ is the ratio of the slow time scale to the fast time scale. Non-Gaussian  processes  $L_1^{{\alpha _1}},L_2^{{\alpha _2}}$  (with $1<\alpha_1, \alpha_2 <2$) are a cylindrical $\alpha_1$-stable process and $\alpha_2$ stable process defined by the orthogonal expansion,respectively,
\begin{equation}
\begin{aligned}
L_1^{{\alpha _1}}:&=\sum_{k=1}^{\infty}\beta_{k}L_{k}(t)e_{k}, \\
L_2^{{\alpha _2}}:&=\sum_{k=1}^{\infty}q_{k}Z_{k}(t)e_{k}, \\
\end{aligned}
\end{equation}
where $\{e_{k}\}_{k\geq 1}$ is an orthonormal  basis of $H$, $\{L_{k}(t)\}_{k\geq 1}$ and $\{Z_{k}(t)\}_{k\geq 1}$ are sequences of independent and identically distributed real-value symmetric $\alpha_1$-stable processes and $\alpha_2$-stable processes defined on the stochastic basis $(\Omega, \mathcal{F}, \mathcal{F}_{t}, \mathbb{P})$, respectively, and $\beta_{k}, q_{k}> 0$ for each $k\geq 1$.

\section{Appendix}
In this Appendix, we   present   materials about the existence of a random invariant manifold. Firstly, we recall  the basic definition of stationary solution.
\begin{definition}
A random variable $\omega\mapsto y(\omega)$ with values in $H$ is called a stationary orbit (or random
fixed point) for a random dynamical system $\phi$ if
\begin{equation}
\phi(t,\omega,y(\omega))=y(\theta_{t}\omega), ~~a.s., \ \text{ for }\ t\in\mathbb{R^{+}},\ \omega\in\Omega.
\end{equation}
\end{definition}
In the following, we present an example to illustrate the abstract definition of stationary orbit.
\begin{example}(Stationary orbit for a Langevin equation) Consider a SDE
\begin{equation}
\label{orbito}
dX_t=-X_tdt+dL^{\alpha}_t, X_0=x.
\end{equation}
This SDE defines a random dynamical asystem
\begin{equation}
\phi(t,\omega,x)=e^{-t}x+\int_{0}^{t}e^{-(t-s)}dL^{\alpha}_s(\omega).
\end{equation}
A stationary orbit of this random dynamical system is
\begin{equation}
y(\omega)=\int_{-\infty}^{0}e^{s}dL^{\alpha}_s(\omega).
\end{equation}
Indeed, we have $\phi(t,\omega,y(\omega))=y(\theta_{t}\omega)$, i.e., $y(\omega)=\int_{-\infty}^{0}e^{s}dL^{\alpha}_s(\omega)$ is a stationary orbit for the random dynamical system \eqref{orbito}.

\end{example}
\begin{lemma}
\label{lemma1}
Under hypothesis $(\mathbf{H.1})$ and $ (\mathbf{H.2})$, the following linear stochastic differential equations
\begin{subequations}
\label{ori}
\begin{equation}
\label{oria}
d\xi^{\varepsilon}(t)=\frac{1}{\varepsilon} A\xi^{\varepsilon}(t) dt+\frac{\sigma_1}{\varepsilon^{\frac{1}{\alpha_1}}}dL_t^{\alpha_1}(\omega_1),~~\xi(0)=\xi_{0},
\end{equation}
\begin{equation}
\label{orib}
d\eta(t) =B\eta(t) dt+\sigma_2dL_t^{\alpha_2}(\omega_2), ~~ \eta(0)=\eta_0.
\end{equation}
\end{subequations}
have c\`{a}dl\`{a}g stationary orbits $\xi^{1, \frac{1}{\varepsilon}}(\omega_1)$ and $\eta(\omega_2)$ , respectively.
\begin{equation}
\left\{
\begin{aligned}
\xi^{1,\frac{1}{\varepsilon}}(\omega_1)&=\frac{\sigma_1}{\varepsilon^{\frac{1}{\alpha_1}}}\int_{-\infty}^{0} e^{\frac{-As}{\varepsilon}}dL_{s}^{\alpha_1}(\omega_1),\\
\eta(\omega_2)&=\sigma_2\int_{-\infty}^{0}e^{-Bs}dL_s^{\alpha_2}(\omega_2).
\end{aligned}
\right.
\end{equation}
\end{lemma}
\begin{proof}
For $ \forall$  $ t\geq 0 $,
the SDE \eqref{orib} has unique c\`{a}dl\`{a}g solution
\begin{equation}\label{orbit}
\varphi(t,\omega_2,\eta_{0})=e^{Bt}\eta_{0}+\sigma_2\int_{0}^{t}e^{B(t-s)}dL_s^{\alpha_2}(\omega_2).
\end{equation}
Then we have
\begin{align*}
\varphi(t,\omega_2,\eta(\omega_2))&=e^{Bt}\eta(\omega_2)+\sigma_2\int_{0}^{t}e^{B(t-s)}dL_s^{\alpha_2}(\omega_2)\\
       &=\sigma_2e^{Bt}\int_{-\infty}^{0}e^{-Bs}dL_s^{\alpha_2}(\omega_2)+\sigma_2\int_{0}^{t}e^{B(t-s)}dL_s^{\alpha_2}(\omega_2)\\
       &=\sigma_2\int_{-\infty}^{t}e^{B(t-s)}dL_s^{\alpha_2}(\omega_2).
\end{align*}
On the other hand,
\begin{equation}
\begin{aligned}
\eta(\theta^2_{t}\omega_2)&=\sigma_2\int_{-\infty}^{0}e^{-Bs}dL_s^{\alpha_2}(\theta^2_t\omega_2) \\
&=\sigma_2\int_{-\infty}^{0}e^{-Bs}d\big(L_{t+s}^{\alpha_2}(\omega_2)-L_{t}^{\alpha_2}(\omega_2)\big) ~~~a.s.\\
 &=\sigma_2\int_{-\infty}^{0}e^{-Bs}dL_{t+s}^{\alpha_2}(\omega_2)  ~~~a.s.\\
 &=\sigma_2\int_{-\infty}^{t}e^{B(t-s)}dL_s^{\alpha_2}(\omega_2)\\
 &=\varphi(t,\omega_2,\eta(\omega_2)).
\end{aligned}
\end{equation}
Thus $\eta(\omega_2)$ is a stationary orbit for a random dynamical system $ \varphi$. By the same way, we have $ \xi^{1, \frac{1}{\varepsilon}}(\omega_1) $ is a stationary orbit for a random dynamical system generated by \eqref{oria}.
\end{proof}

\begin{remark}
Here we emphasize $``-B "$  has no eigenvalue on the imaginary axis.
\end{remark}

\begin{remark}
The process $ (t, \omega_1)\rightarrow \xi^{1, \frac{1}{\varepsilon}}(\theta^1_{\e t} {\omega_1})$  has the same distribution as the process $(t, \omega_1)\rightarrow \xi^{1, 1}(\theta^1_{ t} {\omega_1}) $  by the scale property of  $ \alpha $- stable process.
\end{remark}

Secondly, we provides some important pathwise properties for L\'evy process  with two-sided time $ t \in \mathbb{R}$, which comes from \cite[Lemma 1]{Xian}
\begin{lemma}(pathwise boundedness and convergence)\\
Let $L_t$ be a two-sided L\'evy process on $\mathbb{R}^{n}$ for which $\mathbb{E}|L_1| \textless \infty $ and $\mathbb{E}|L_1|= M$. Then
we have the following
\begin{enumerate}
\item[(1)]$\lim_{t\rightarrow \pm \infty }(1/t)L_t=M $,  a.s.;
\item[(2)]The integrals $\int_{-\infty}^{t}e^{-\lambda (t-s)}dL_s(\omega)$ are bounded in $\lambda \geq 1$ on finite time intervals $[T_1,T_2]$.
\end{enumerate}
In the following, we will give an example to illustrate the second conclusion of Lemma 1.
\begin{example}
The integral $\int_{-\infty}^{t}e^{-(t-s)}dB_s(\omega)$ is bounded on finite time intervals $[-1, 1]$.
\end{example}
\begin{proof}
Here we take $\lambda=1$, by the It\^o's isometry formula, then we have
\begin{equation}
\begin{aligned}
\mathbb{E}\left(\int_{-\infty}^{t}e^{-(t-s)}dB_s(\omega)\right)^{2}&=\int_{-\infty}^{t}e^{-2(t-s)}ds\\
&:=\frac{1}{2}e^{-2t}
\end{aligned}
\end{equation}
Obviously, $\frac{1}{2}e^{-2t}$ is a continuous function about $t$, it is bounded on finite time intervals $[T_1,T_2]$. Therefore the integral $\int_{-\infty}^{t}e^{-(t-s)}dB_s(\omega)$ is bounded on finite time intervals $[-1, 1]$.
\begin{equation}
\end{equation}
\end{proof}
\end{lemma}

In fact, we have to find a modification of L\'evy process $ L_t$, such that the cocycle property is satisfied for every $ \omega \in \Omega $. The following lemma covers the perfection problem for all processes of this work (see \cite{Kai}).
\begin{lemma}
Let $ H $ be a separable  Banach space, $ (S_t)_{t \in \mathbb{R}}$ is a $ H $ -valued and $ \mathcal{F}-$ measurable stochastic process with c\'adl\'ag paths generating a crude cocycle with respect to the metric dynamical system $(\Omega, \mathcal{F}, \mathbb{P}, (\theta_t)_{t\in \mathbb{R}})$, i.e., for all $ t \in \mathbb{R}$ we have
\begin{equation}
S_t=S_0\circ \theta_t, ~~~~\mathbb{P}-a.s.
\end{equation}
Then there is an $ H $ -valued process $ \hat{S}= (\hat{S}_t)_{t \in {\mathbb{R}}}$, such that:
\begin{enumerate}
\item[(i)]The process $S$ and $\hat{S} $ are undistinguishable;
\item[(ii)]The process $\hat{S}$ is strictly stationary, i.e.
\begin{equation}
\hat{S}_t(\omega) =\hat{S}_0(\theta_t \omega)
\end{equation}
for all $ t \in \mathbb{R}, \omega \in \Omega$.
\end{enumerate}
\end{lemma}
In the following, we will present an example to explain why don't we take $\hat{S}_t=S_t$ ?
\begin{example}
Let $(\Omega, \mathscr{F})=(\mathbb{R}, \mathcal{B(\mathbb{R})} )$ and $\mathbb{P}$ be a probability measure which is equivalent to the Lebesgue-measure
and $\theta_t(\omega)=\omega+t$. Now we define
\begin{eqnarray}
\varphi(t,\omega):=
\begin{cases}
\frac{sin(\theta_t(\omega))}{sin(\omega)}, ~~& sin(\omega)\neq 0; \\
1, ~~& sin(\omega)=0,
\end{cases}
\end{eqnarray}
then $\varphi(0, \omega)=1$ for each $\omega\in \Omega$, and $\varphi$ forms a measurable crude multiplicative cocycle. Indeed, fix $s\in \mathbb{R}$ and $\omega \in \Omega \backslash \{ (\pi\mathbb{Z}\cup (\pi\mathbb{Z}-s))\}$, we get
\begin{equation}
\varphi(t+s, \omega)=\frac{sin(\omega+t+s)}{sin(\omega+s)}\cdot \frac{sin(\omega+s)}{sin(\omega)}=\varphi(t,\theta_s\omega)\varphi(s,\omega)
\end{equation}
for each $t\in \mathbb{R}$.

Now, we set $\omega \in (0, \pi)$, $s=\pi-\omega$ and $t:=\frac{\pi}{2}$, which implies $\varphi(s,\omega)=0$ and $\varphi(t+s,\omega)\neq 0$. In this case, there is no perfect cocycle $\psi$ which is still indistinguishable form $\varphi$. if we assume that $\varphi(\cdot, \omega)=\psi(\cdot, \omega)$ for each $\omega \in \Omega_{1}\subset \Omega$, then
\begin{equation}
\varphi(t+s, \omega)=\psi(t+s, \omega)=\psi(t, \theta_{s}\omega)\psi(s,\omega)=\psi(t,\theta_{s}\omega)\varphi(s,\omega)=0
\end{equation}
for each $t\in \mathbb{R}$ and $\omega \in \Omega_{1}$.

Since $\varphi(t+s, \omega)\neq 0$ for $t=\frac{\pi}{2}$ and $\omega \in (0, \phi)$, then we have
\begin{equation}
\begin{aligned}
(0,\pi)\cap \Omega_{1}&=\emptyset, \\
\mathbb{P}(\Omega_1)&< 1,
\end{aligned}
\end{equation}
which implies that there is no indistinguishable perfect cocycle for $\varphi$.
\end{example}

%

In the next, we examine the following nonstandard boundary value problem, for any $ 0 \leq t \leq T $,
\begin{equation}
\label{rrde55}
\left\{
\begin{aligned}
\frac{d\bar{X}^{\varepsilon}(t)}{dt}&= A\bar{X}^{\varepsilon}(t)+ F\big(\bar{X}^{\varepsilon}(t)+\xi^{1, 1}(\theta^1_{ t}\omega_1),\bar{Y}^{\varepsilon}(t)+\eta(\theta^2_{\e t}{\omega_2})\big), \bar{X}^{\varepsilon}(0)=\gamma(\bar{Y}^{\varepsilon}(0)), \\
\frac{d\bar{Y}^{\varepsilon}(t)}{dt}&= \varepsilon B\bar{Y}^{\varepsilon}(t)+\varepsilon G \big(\bar{X}^{\varepsilon}(t)+\xi^{1, 1}(\theta^1_{ t}\omega_1),\bar{Y}^{\varepsilon}(t)+\eta(\theta^2_{\e t}{\omega_2})\big), \bar{Y}^{\varepsilon}(T)=\tilde{y},\\
\end{aligned}
\right.
\end{equation}
where $ \varepsilon $ is a small positive parameter, $ \tilde{y} \in  \mathbb{R}^{m}, T \textgreater 0 $ and $ \gamma \in Lip(\mathbb{R}^{m},\mathbb{R}^{n})$ are given.

\begin{lemma}
\label{lemma1}
Assume that  the hypotheses $\mathbf{(H.1)}$-$\mathbf{(H.4)}$ hold.  Then for any ~$ \tilde{y} \in {\mathbb{R}^{m}}, T \textgreater 0, \gamma \in Lip(\mathbb{R}^{m},\mathbb{R}^{n})$, there exists a sufficient
small positive number $ \delta $,  such that for ~$ 0 \textless \varepsilon \textless \delta $,  the random dynamical system
defined by  \eqref{rrde55} has a unique solution $ (\bar{X}^{\varepsilon}(t,\omega, T, \gamma, \tilde{y}),\bar{Y}^{\varepsilon}(t,\omega, T, \gamma, \tilde{y}) )$.
\end{lemma}
\begin{proof}
Let $ V_{ A}^{\varepsilon}(t, \omega)$ and $ V_{\varepsilon B}^{\varepsilon}(t, \omega)$ be the fundamental solution of the linear system
\begin{equation}
\left\{
\begin{aligned}
\frac{d\bar{X}^{\varepsilon}(t)}{dt}&=  A\bar{X}^{\varepsilon}(t),\\
\frac{d\bar{Y}^{\varepsilon}(t)}{dt}&=\varepsilon B\bar{Y}^{\varepsilon}(t).
\end{aligned}
\right.
\end{equation}
Then the boundary value problem \eqref{rrde55} is equivalent to the system of integral equations
\begin{equation}
\label{rrde3}
\left\{
\begin{aligned}
\bar{X}^{\varepsilon}(t)&=V_{ A}^{\varepsilon}(t)\gamma(\bar{Y}^{\varepsilon}(0))+\int_{0}^{t}V_{ A}^{\varepsilon}(t-s)F(\theta^{\e}_{s}{\omega},\bar{X}^{\varepsilon}(s),\bar{Y}^{\varepsilon}(s))ds ,      \\
\bar{Y}^{\varepsilon}(t)&= V^{\e}_{\varepsilon B}(t-T)\tilde{y}+\varepsilon \int_{T}^{t}V^{\e}_{\varepsilon B}(t-s)G(\theta^{\e}_{s}{\omega}, \bar{X}^{\varepsilon}(s),
\bar{Y}^{\varepsilon}(s))ds.
\end{aligned}
\right.
\end{equation}
To study \eqref{rrde3}, we introduce the following spaces
\begin{equation}
\mathcal{C}_1:=C([0,T], \mathbb{R}^{n}), ~~ \mathcal{C}_2:=C([0,T], \mathbb{R}^{m})
\end{equation}
and endow these spaces with the following norms
\begin{equation}
\begin{aligned}
||x||_{1, \beta}&:=\max_{0\leq t \leq  T}e^{-\beta(T-t)}|x(t)|,  ~~ for  ~~x \in \mathcal{C}_1, \\
||y||_{2, \beta}&:=\max_{0\leq t \leq  T}e^{-\beta(T-t)}|y(t)|,  ~~ for  ~~y \in \mathcal{C}_2. \\
\end{aligned}
\end{equation}
Let $ \mathcal{C} $ be the product space $\mathcal{C}:=\mathcal{C}_1 \times \mathcal{C}_2 $, $ z=(x,y)\in \mathcal{C} $. $ \mathcal{C} $ equipped with the norm
\begin{equation}
||z||_{\beta}:=||x||_{1, \beta}+||y||_{2, \beta},
\end{equation}                                                                                           %
is a Banach space.

Introduce the following two operators $\mathcal{J}_{1}^{\varepsilon}:\mathcal{C} \rightarrow \mathcal{C}_1$ and $\mathcal{J}_{2}^{\varepsilon}:\mathcal{C} \rightarrow \mathcal{C}_{2}$ by
\begin{equation}
\begin{aligned}
  \tilde{x}(t)&=\mathcal{J}_{1}^{\varepsilon}\big(z(\cdot)\big)[t]:=V_{ A}^{\varepsilon}(t)\gamma(\bar{Y}^{\varepsilon}(0))+\int_{0}^{t}V_{ A}^{\varepsilon}(t-s)F(\theta^{\e}_{s}{\omega},\bar{X}^{\varepsilon}(s),\bar{Y}^{\varepsilon}(s))ds,\\[1mm]
  \tilde{y}(t)&=\mathcal{J}_{2}^{\varepsilon}\big(z(\cdot)\big)[t]:=V^{\e}_{\varepsilon B}(t-T)\tilde{y}+\varepsilon \int_{T}^{t}V^{\e}_{\varepsilon B}(t-s)G(\theta^{\e}_{s}{\omega}, \bar{X}^{\varepsilon}(s),
\bar{Y}^{\varepsilon}(s))ds.
\end{aligned}
\end{equation}
Define the mapping  $\mathcal{J}^{\varepsilon}$ given by
\begin{equation}
\tilde{z}\big(\cdot\big)=\mathcal{J}^{\varepsilon}\big(z(\cdot)\big)
= \left(
  \begin{array}{ccc}
  \mathcal{J}_{1}^{\varepsilon}(z(\cdot))[t]\\
  \mathcal{J}_{2}^{\varepsilon}(z(\cdot))[t]\\
  \end{array}
\right).
\end{equation}
It is obvious that a fixed point $z ^{\e}$ of $ \mathcal{J}^{\varepsilon} $ represents  a solution of the boundary value problem \eqref{rrde55}.
Under hypothesis $\mathbf{H.1}$-$\mathbf{H.3}$, $\mathcal{J}^{\varepsilon}$ maps $ \mathcal{C}$ into itself,
and there are constants $ a_1 \geq 1 $ and $ a_2 \geq 1 $ such that
\begin{equation}
\begin{aligned}
||V^{\varepsilon}_{A}(t-s)||& \leq a_1 e^{-M_{A}(t-s)}, ~for  ~~t \geq s,  \\
||V^{\varepsilon}_{\varepsilon B}(t-s)||& \leq a_2 e^{\varepsilon \|B\||t-s|}, ~for ~any ~~t, s. \\
\end{aligned}
\end{equation}
In the following, we show that $ \mathcal{J}^{\varepsilon} $ is also strictly contractive. Let
\begin{equation}
\triangle {\tilde{x}}:=\tilde{x}_1-\tilde{x}_2,~~~\triangle {\tilde{y}}:=\tilde{y}_1-\tilde{y}_2, ~~~ |\triangle \tilde{z}|:=|\triangle \tilde{x}|+|\triangle \tilde{y}|.
\end{equation}
Let $ (\tilde{x}_1(t),\tilde{y}_1(t)) $ and $ (\tilde{x}_2(t),\tilde{y}_2(t) )$ satisfy the equations
\eqref{rrde3}.  Then we have
\begin{equation}
\label{xx}
\begin{aligned}
|\triangle \tilde{x}(t)|&=\big|V_{ A}^{\varepsilon}(t)\gamma(\tilde{y}_1^{\varepsilon}(0))+\int_{0}^{t}V_{ A}^{\varepsilon}(t-s)F(\theta^{\e}_{s}{\omega},\tilde{x}^{\varepsilon}_1(s),\tilde{y}^{\varepsilon}_1(s))ds. \\
&-V_{ A}^{\varepsilon}(t)\gamma(\tilde{y}_2^{\varepsilon}(0))-\int_{0}^{t}V_{ A}^{\varepsilon}(t-s)F(\theta^{\e}_{s}{\omega},\tilde{x}^{\varepsilon}_2(s),\tilde{y}^{\varepsilon}_2(s))ds\big| \\
&\leq a_1 e^{-M_{A}t} ||\gamma||_{Lip}|\triangle {\tilde{y}}^{\varepsilon}(0)|+a_1 C \int_{0}^{t}e^{-M_{A}(t-s)}|\triangle \tilde{z}(s)|ds \\
& \leq a_1 e^{-M_{A}t} ||\gamma||_{Lip}|\triangle {\tilde{y}}^{\varepsilon}(0)|+a_1 C ||\triangle \tilde{z}||_{\beta}e^{-M_{A}t}e^{\beta T}
\int_{0}^{t}e^{(M_{A}-\beta)s}ds \\
& \leq a_1 e^{-M_{A}t} ||\gamma||_{Lip}|\triangle{\tilde{y}}^{\varepsilon}(0)|+\frac{a_1 C e^{\beta(T-t)}}{M_{A}-\beta}||\triangle \tilde{z}||_{\beta}.
\end{aligned}
\end{equation}
On the other hand, we get
\begin{equation}
\begin{aligned}
|\triangle \tilde{y}(t)|&=\varepsilon\big|\int_{T}^{t}V_{\varepsilon B}^{\varepsilon}(t-s)G(\theta^{\e}_{s}{\omega},\tilde{x}^{\varepsilon}_1(s),\tilde{y}^{\varepsilon}_1(s))ds-\int_{T}^{t}V_{ \varepsilon B}^{\varepsilon}(t-s)G(\theta^{\e}_{s}{\omega},\tilde{x}^{\varepsilon}_2(s),\tilde{y}^{\varepsilon}_2(s))ds\big| \\
&\leq \varepsilon  a_2 C \int_{T}^{t}e^{\varepsilon \|B\|(t-s)}|\triangle z(s)|ds \\
& \leq \varepsilon  a_2 C ||\triangle z||_{\beta}e^{\beta T-\varepsilon \|B\|t}\int_{t}^{T}e^{-(\beta-\varepsilon \|B\|)s}ds.
\end{aligned}
\end{equation}
We assume  $ \varepsilon $ to be sufficient small such that
\begin{equation}
\beta-\varepsilon \|B\| \textgreater 0.
\end{equation}
Then, we get
\begin{equation}
\label{y}
||\triangle \tilde{y}(t)||_{2,\beta}\leq \frac{\varepsilon a_2 C }{\beta-\varepsilon \|B\|}||\triangle z||_{\beta}.
\end{equation}
Bring \eqref{y} into \eqref{xx}, we get
\begin{equation}
\label{x}
||\triangle \tilde{x}||_{1,\beta} \leq \big(\frac{\varepsilon a_1 a_2 ||\gamma||_{Lip}C}{\beta-\varepsilon \|B\|}+\frac{a_1 C}{M_{A}-\beta}\big)||\triangle z||_{\beta}
\end{equation}

Combining \eqref{x} with \eqref{y}, we have
\begin{equation}
||\triangle \tilde{z}||_{\beta} \leq \rho(\varepsilon)||\triangle z||_{\beta},
\end{equation}
where
\begin{equation*}
\rho(\varepsilon)=\frac{\varepsilon a_2 C}{\beta-\varepsilon \|B\|}(1+a_1 ||\gamma||_{Lip})+\frac{a_1 C}{M_{A}-\beta}.
 \end{equation*}
Note that,
\begin{equation}
\rho'(\varepsilon)=\frac{a_2C\beta}{(\beta-\varepsilon \|B\|)^2} (1+a_1 ||\gamma||_{Lip}) > 0,  ~~~~\rho(0) \textless 1.
\end{equation}
Then there is a sufficiently small constant $\delta>0$ and a constant $\rho_{0}\in(0,1)$, such that
\begin{equation*}
0<\rho(\varepsilon)\leq\rho_{0}<1,\ ~\mbox{for}~\ \varepsilon \in (0,\delta),
\end{equation*}
which implies that $\mathcal{J}^{\varepsilon}$ is strictly contractive. Therefore the system \eqref{rrde55} has a unique solution $ \big(\bar{X}^{\varepsilon}(t,\omega, T, \gamma, \tilde{y}),\bar{Y}^{\varepsilon}(t,\omega, T, \gamma, \tilde{y}) \big)$.
\end{proof}
\begin{remark}
There is a one-to-one relation between $\tilde{y}$ and $ \bar{Y}^{\varepsilon} (0,\omega, T, \gamma, \tilde{y})=y_0$.  Since the operator $ \mathcal{J}$ explicitly depends on $ \gamma $, we use in the following the notation $ \mathcal{J}^{\varepsilon}_{\gamma}$. Moreover, for  any $ \gamma \in \mathcal{L}$ and $ t=T$, the map
\begin{equation}
\mathbb{R}^{m}\ni y_{0}\rightarrow \phi_2^{\varepsilon}\big(t,\theta^{\varepsilon}_{-t}, (\gamma(y_0),y_0)\big)=\tilde{y}\in\mathbb{R}^{m}
\end{equation}
is invertible, and the inverse mapping $\xi^{\varepsilon}$ is given by
\begin{equation}
\xi^{\varepsilon}(T, \theta^{\e}_{T}\omega, \gamma)(\tilde{y})=\bar{Y}^{\varepsilon} (0,\omega, T, \gamma, \tilde{y}).
\end{equation}
\end{remark}
In the following lemma, we will illustrate the path property and dependence on the function $\gamma$.
\begin{lemma}
\label{lemmaa}
Assume that the hypotheses $(\mathbf{H.1})$-$(\mathbf{H.4})$ hold.  Then for  given $ \omega \in \Omega, T > 0, \tilde{y} \in \mathbb{R}^{m},  t \in [0, T]$ and for sufficiently small $ \varepsilon $, the solution $ z^{\varepsilon}(t, \omega, T, \gamma, \tilde{y})$ of
\eqref{rrde55} depends Lipschitz continuously on $ \gamma $.
\begin{equation}
\sup_{\tilde{y}\in \mathbb{R}^{m}}||z^{\varepsilon}(\cdot, \omega, T, \gamma_1, \tilde{y})-z^{\varepsilon}(\cdot, \omega, T, \gamma_2, \tilde{y})||_{\beta}\leq \frac{a e^{-\beta T}}{1-\rho(\varepsilon)}||\gamma_1-\gamma_2||_{\infty}.
\end{equation}
\begin{proof}
For $ \gamma_1 $ and $ \gamma_2 $   any functions in $ \mathcal{L}$,  define
\begin{equation}
\triangle_{\gamma} z^{\varepsilon}(t):=z^{\varepsilon}(t, \omega, T, \gamma_1, \tilde{y})-z^{\varepsilon}(t, \omega, T, \gamma_2, \tilde{y}).
\end{equation}
By Lemma \ref{lemma1},  with fixed $(\omega, T, \gamma, \tilde{y})$,   there exists a unique fixed point $z^{\varepsilon}$ of the operator $\mathcal{J}^{\varepsilon}_{\gamma}$. Thus,  we get
\begin{equation}
\begin{aligned}
||\triangle_{\gamma}z^{\varepsilon}||_{\beta}&=||\mathcal{J}^{\varepsilon}_{\gamma_1}z^{\varepsilon}(\cdot, \omega, T, \gamma_1, \tilde{y})-\mathcal{J}^{\varepsilon}_{\gamma_2}z^{\varepsilon}(\cdot, \omega, T, \gamma_2, \tilde{y})||_{\beta}\\
& \leq||\mathcal{J}^{\varepsilon}_{\gamma_1}z^{\varepsilon}(\cdot, \omega, T, \gamma_1, \tilde{y})-\mathcal{J}^{\varepsilon}_{\gamma_1}z^{\varepsilon}(\cdot, \omega, T, \gamma_2, \tilde{y})||_{\beta} \\
&+||\mathcal{J}^{\varepsilon}_{\gamma_1}z^{\varepsilon}(\cdot, \omega, T, \gamma_2, \tilde{y})-\mathcal{J}^{\varepsilon}_{\gamma_2}z^{\varepsilon}(\cdot, \omega, T, \gamma_2, \tilde{y})||_{\beta} \\
& \leq \rho(\varepsilon)||\triangle_{\gamma}z^{\varepsilon}||_{\beta}+||V^{\varepsilon}_{A}(\cdot, \omega)||_{\beta}
||\gamma_1-\gamma_2||_{\infty}.
\end{aligned}
\end{equation}
By hypothesis $(\mathbf{H.3})$, we have
\begin{equation}
||\triangle_{\gamma}z^{\varepsilon}||_{\alpha}\leq \frac{a_1 e^{\alpha T}}{1-\rho(\varepsilon)}||\gamma_1-\gamma_2||_{\infty}.
\end{equation}
This proof is complete.
\end{proof}
\end{lemma}
Define the random graph transform by
\begin{equation}
\label{main}
\begin{aligned}
\psi^{\varepsilon}(T, \omega, \gamma)(\tilde{y})&:= \phi^{\varepsilon}_1\big(T, \omega, \big(\gamma(\xi^{\varepsilon}(T, \gamma^{\varepsilon}_{T}\omega, \gamma)(\tilde{y})),\xi^{\varepsilon}(T, \gamma^{\varepsilon}_{T}\omega, \gamma)(\tilde{y}\big)\big)\\
&=\bar{X}^{\varepsilon}(T, \omega, T, \gamma, \tilde{y}) \\
&=V^{\varepsilon}_{A}(T)\gamma(\bar{Y}^{\varepsilon}(0))+\int_{0}^{T}V^{\varepsilon}_{A}(t-s)F\left(\theta^{\varepsilon}_{s}\omega,\bar{X}^{\varepsilon}(s), \bar{Y}^{\varepsilon}(s)\right)ds.
\end{aligned}
\end{equation}
such that the cocycle property is satisfied. By the similar techniques \cite[Lemma 4.1]{BS} and \cite[Lemma 4.7]{BS}, we get the following two lemmas.
\begin{lemma}
\label{positive}
For $ \omega \in \Omega $, let $ \gamma^{\varepsilon}(\omega, \cdot) \in Lip(\mathbb{R}^m, \mathbb{R}^{n} )$. Suppose that  for  $ \gamma^{\varepsilon}(\omega, \cdot)$
  a random fixed points of $ \psi^{\varepsilon}$, we have
\begin{equation}
\psi^{\epsilon}(t,\omega,\gamma^{\varepsilon}(\omega, \cdot) )(\tilde{y})=\gamma^{\varepsilon}(\theta^{\epsilon}_t\omega, \tilde{y}),
\end{equation}
for $ t \textgreater 0 $  and $ \omega \in \Omega$.
Then the random Lipschitz manifold defined by
\begin{equation}
\mathcal{M}^{\epsilon}(\omega):=\{(\gamma^{\varepsilon}(\omega, \tilde{y}),\tilde{y})|\tilde{y}\in \mathbb{R}^{m}\},
\end{equation}
is positively invariant.
\end{lemma}

\begin{lemma}
Assume that the hypotheses $(\mathbf{H.1})$-$(\mathbf{H.4})$ hold. Then for sufficiently small $ \varepsilon $ and sufficiently large $ T $, the graph transform $ \psi^{\varepsilon}(T, \omega, \cdot)$ maps the set $ \mathcal{L}_{\kappa}$ into itself, where $ \kappa $ is any positive number satisfying
\begin{equation}
\kappa\geq \kappa^{*}=\frac{a_2}{1-\beta_0-a_1a_2e^{-\frac{T\beta}{2}}},
\end{equation}
where $\beta_0$ is any given number from the interval $\left(\frac{a_1 C}{M_{A}-\beta}, 1\right)$.
\end{lemma}
Set
\begin{equation}
\mathscr{H}^{\varepsilon}(\omega):=\bigcup_{t\geq 0}\psi^{\varepsilon}(t, \theta^{\e}_{-t}\omega, L_{{\kappa}^{*}}).
\end{equation}
Then we have
\begin{equation}
\psi^{\varepsilon}(t,\omega, \mathscr{H}^{\varepsilon}(\omega))\subset \mathscr{H}^{\varepsilon}(\theta^{\varepsilon}_{t}\omega), ~~t\geq 0,
\end{equation}
By a similar technique as  in \cite[Lemma 4.10]{BS}, we are able to prove the cocycle property for the graph transform $\psi^{\varepsilon}$. By means of the graph transform $\psi^{\varepsilon}$, we define an operator $\mathcal{A}^{\varepsilon}$  via
\begin{equation}
\label{define}
(\omega, \tilde{y})\rightarrow \mathcal{A}^{\varepsilon}(\gamma)(\omega, \tilde{y}):=\psi^{\varepsilon}(T, \theta^{\varepsilon}_{-T}\omega,\gamma(\theta^{\varepsilon}_{-T}\omega))(\tilde{y}).
\end{equation}

Now we prove the exponential tracking property.
\begin{lemma}
\label{lemfix}
Assume that the hypotheses $(\mathbf{H.1})$-$\mathbf{(H.4)}$ hold. Then for sufficiently small $ \varepsilon $, the Lipschitz invariant manifold has the exponential tracking property in the following sense: \\
For every  solution $\bar{z}^{\varepsilon}(t,\omega)=(\bar{x}^{\varepsilon}(t,\omega), \bar{y}^{\varepsilon}(t,\omega))$ for \eqref{rrde55}, there is an orbit ${\bar{\bar{z}}}^{\varepsilon}(t,\omega)=(\bar{\bar{x}}^{\epsilon}(t,\omega), \bar{\bar{y}}^{\epsilon}(t,\omega))$ on the manifold $\mathcal{M}^{\epsilon}(\omega)$ which satisfies the evolutionary equation
\begin{equation}
\dot{\tilde{y}}^{\varepsilon}=B\tilde{y}^{\varepsilon}+G(\gamma^{\varepsilon}(\theta^{\varepsilon}_t\omega, \tilde{y}^{\varepsilon}),\tilde{y}^{\varepsilon},\theta^{\varepsilon}_t\omega)
\end{equation}
such that
\begin{equation}
||{{\bar{z}}}^{\varepsilon}(t,\omega)-{\bar{\bar{z}}}^{\varepsilon}(t,\omega)||_{\infty}\leq C_{\tilde{\kappa}, \varepsilon}e^{\frac{-\tilde{\kappa} t}{\varepsilon}}||\bar{z}_0-\bar{\bar{z}}_0||_{\infty}
\end{equation}
with $\bar{z}_0=(\bar{x}^{\varepsilon}(0),\bar{y}^{\varepsilon}(0))$, $\bar{\bar{z}}_0=(\bar{\bar{x}}^{\varepsilon}(0),\bar{\bar{y}}^{\varepsilon}(0))$ and some positive constant $\tilde{\kappa}$.
\end{lemma}
\begin{proof}
Thanks to  \eqref{main}, we can use dominated convergence theorem and the same method  in \cite[Theorem 4.2]{Fu} to obtain  the required results.
\end{proof}

\medskip
\textbf{Acknowledgements}.  We would like to thank Xianming Liu (Huazhong University of Sciences and Technology, China) for helpful discussions.

\end{document}